\documentclass{amsart}
\usepackage{amssymb, amsmath, amsthm, graphics, comment, xspace, enumerate}
\newtheorem{thm}{Theorem}[section]

\newtheorem{lem}[thm]{Lemma}
\newtheorem{prop}[thm]{Proposition}
\theoremstyle{definition}
\newtheorem{defn}[thm]{Definition}
\newtheorem{example}[thm]{Example}
\theoremstyle{remark}

\numberwithin{equation}{section}

\newcommand{\DD}{\mathcal D}

\begin{document}
\title[Stationary dense operators]{Stationary dense operators in sequentially complete locally convex spaces}

\author{Daniel Velinov}
\address{Department for Mathematics, Faculty of Civil Engineering, Ss. Cyril and Methodius University, Skopje,
Partizanski Odredi
24, P.O. box 560, 1000 Skopje}
\email{velinovd@gf.ukim.edu.mk}

\author{Marko Kosti\' c}
\address{Faculty of Technical Sciences,
University of Novi Sad,
Trg D. Obradovi\' ca 6, 21125 Novi Sad, Serbia}
\email{marco.s@verat.net}

\author{Stevan Pilipovi\' c}
\address{Department for Mathematics and Informatics,
University of Novi Sad,
Trg D. Obradovi\' ca 4, 21000 Novi Sad, Serbia}
\email{pilipovic@dmi.uns.ac.rs}

{\renewcommand{\thefootnote}{} \footnote{
\\ \text{  }  \ \    {\it Key words and phrases.} stationary dense operators, distribution semigroups, ultradistribution semigroups, well-posedness, locally convex spaces.}}

{\renewcommand{\thefootnote}{} \footnote{ $\,^*$ This research is partially supported by grant 174024 of Ministry
of Science and Technological Development, Republic of Serbia.}}

\begin{abstract}
The purpose of this paper is to investigate the stationary dense operators and their connection to distribution semigroups and abstract Cauchy problem in sequentially complete spaces.
\end{abstract}
\maketitle

\section{Introduction and Preliminaries}
In \cite{li121} and \cite{ku112} are given definitions of a distribution semigroups with densely and non densely defined generators.  Stationary dense operators in a Banach space $E$ and their connections to densely defined distribution semigroups are introduced and analyzed in \cite{ku101}. Our main aim in this paper is to consider such operators when $E$ is sequentially complete locally convex space.\\
\indent
When we are dealing with a semigroup in a locally convex space, in order to have the existence of the resolvent of an infinitesimal generator $A$,  we suppose that the semigroup is equicontinuous.
Moreover for $A$, we are led to study the behavior of $A_{\infty}$ in $D_{\infty}(A)$. If $A$ is stationary dense, the information lost in passing from $A$ in $E$ to $A_{\infty}$ in $D_{\infty}(A)$, can be retrieved.
Let us emphasize that the interest for the stationarity of $A$ comes out from the solvability of the Cauchy problem $u'=Au$, $u(0)=u_0$, for every $u\in D(A^n)$ for some $n$. This is noted in Proposition \ref{pc-kunst-isto}. In Theorem \ref{thmmain} we have a partial answer to the converse question: Whether the solvability for every $x\in D(A^n)$ implies $n$-stationarity. Also, we extend the results of \cite{fujiwara}, \cite{ku101}, \cite{ush} and give new examples of stationary dense operators on sequentially complete locally convex spaces. \\
\indent Here we denote with $E$ a sequentially complete locally convex space with system of seminorms $\circledast$. For a closed linear operator $A$ we define $D_{\infty}(A)=\bigcap_{n=1}^{\infty}D(A^nx)$, for $x\in E$. Then the space $D_{\infty}(A)$ is Fr\'echet space as a projective limit of Fr\'echet spaces with seminorms $p_n(x)=\sup\limits_{p\in\circledast}\sum\limits_{k=0}^{n}p(A^kx)$, for  for all $q\in\circledast$ and $x\in D_{\infty}(A)\subseteq E$, $n\in{\mathbb N}$.\\
The distribution semigroups are defined by (\cite{li121}):\\
A distribution $G$ is a distribution semigroup (shortly $DSG$) if following conditions are satisfied:
\begin{itemize}
\item[(i)] $G\in\DD'_+(L(E;E)),\quad G=0, \quad \mbox{for}\quad t<0;$
\item[(ii)] $G(\varphi\ast\psi)=G(\varphi)G(\psi)$ for all $\varphi,\psi\in\DD_0$, where $\DD_0$ is the space of all $\varphi\in\DD$ such that $\varphi(t)=0$ for $t<0$;
\item[(iii)] Let $\varphi\in\DD_0$, $x\in E$ and $y=G(\varphi)x$. The distribution $Gy$ is equal to a continuous function $u$ on $[0,+\infty)$ and $u(t)=0$ for $t<0$ with range in $E$ and $u(0)=y$;
\item[(iv)] The range of the all elements $G(\varphi)x$, where $\varphi\in\DD_0$, $x\in E$ is dense in $E$;
\item[(v)] If for all $\varphi\in\DD_0$, $G(\varphi)x=0$, $x\in E$ then $x=0$. \end{itemize}
In further consideration will be used the last definition of distribution semigroups, although often are considered semigroups which are not densely defined (cf. \cite{ku112}).
In order to give the definition of exponential distribution semigroups we give the definition of the spaces through tempered distributions.
Let $a\geq 0.$ Then
$$
{\mathcal{SE}}_{a}({\mathbb R}):=\{\phi \in C^\infty({\mathbb
R}) :  e^{a t}\phi\in{\mathcal S}(\mathbb R)\}.
$$
Define the convergence in this space by
$$\phi_n\rightarrow 0 \mbox{ in } {\mathcal{SE}}_{a}({\mathbb R}) \mbox{ iff }
e^{a \cdot}\phi_n\rightarrow 0 \mbox{ in } {\mathcal
S}({\mathbb R}).$$
Denote by $\mathcal{SE}'_{a}(\mathbb{R},E)$ the space
$L(\mathcal{SE}_{a}(\mathbb{R}),E)$ which is formed from all
continuous linear mappings from $\mathcal{SE}_{a}(\mathbb{R})$
into $E$ equipped with the strong topology.

It holds,
$$
 F\in {\mathcal{SE}}'_{a}(\mathbb{R},E) \mbox{
if and only if}\quad e^{-a\cdot}F\in {\mathcal{S}}'(\mathbb{R},E).
$$
In sequel will be given the structure of the space ${\mathcal{SE}}'_{a}({\mathbb{R}},E)$ used in definition of exponential distribution semigroups.
\begin{prop}\label{novo}
Let $T \in \mathcal{SE}'_{a}(\mathbb{R},E).$ Then there exists
an polynomial $P$ and a function $g\in
C({\mathbb R},E)$ with the property that there exist $k>0$ and
$C>0$, such that
$$
e^{-ax}||g(x)||\leq C |x|^k, \;x\in \mathbb{R}\; \mbox{ and }\;
T=P(d/dt)g.
$$
\end{prop}

The proof of this proposition is similar like in ultradistribution case, even more simple (see \cite{tica}, \cite{hypkpv} for ultradistribution and hyperfunction case, when $E$ is a Banach space), hence will be omitted.\\
A distribution semigroup $G$ is an exponential distribution semigroup,
($E$-distribution semigroup, $(EDSG)$ shortly), if in addition to
$(i)-(v),$ $G$ fulfills:
$$
(vi)\;\;(\exists a\geq 0)  G\in {\mathcal
SE}'_a({\mathbb R},L(E))).\;
$$
The generator $A$ of a distribution semigroup $G$ is defined as
$$A=\{(x,y)\in E^2\, :\, \forall\varphi\in\DD, G({-\varphi}')x=G(\varphi)y\}.$$
Also we want to refer to \cite{li121} and \cite{ku112} for structural theorems for distribution semigroups and exponential semigroups (for ultradistribution case see ~\cite[Theorem 2.7]{tica}).
An $C_0$-semigroup $(T(t))_{t\geq0}$ in sequentially complete locally convex space $E$ is an equicontinuous semigroup if the set $\{T(t)\,:\, t\geq0\}$ is an equicontinuous set in $E$. Moreover, we say that an $C_0$-semigroup is quasi-equicontinuous if there is an $a\geq0$ such that $\{e^{-at}T(t)\, :\, t\geq0\}$ is an equicontinuous set.\\
\indent Let $D$ and $E$ are sequentially complete locally convex spaces and let $P\in\DD^{'}_0(L(D,E))$. Then $G\in{\mathcal D}^{'}_0(L(E,D))$ is said to be a distribution fundamental solution for $P={\delta}'\otimes I-\delta\otimes A$ when $P\ast G=\delta\otimes I_E$ and $G\ast P=\delta\otimes I_D$.
Let $E$ be a sequentially complete locally convex space. In \cite{sch166}, is given structural theorem for $E$-valued distributions and there is an  analogue in the case of $E$-valued tempered distributions which can be proved similar like in ultradistribution case. Next theorem is a structural theorem for distribution semigroups. Actually, a version of this  theorem is given in \cite{tica} for ultradistribution semigroups when $E$ is a Banach space. The proof is similar.
\begin{thm}\label{strucdis}
\begin{itemize}
\item[(a)] $A$ generates a (DSG) $G$.
\item[(a)'] $A$ generates a  (EDSG) $G$.
\item[(b)] $A$ generates a  (DSG) $G$ such that for
every $a>0,$ $G$ is of the form $G =P^{a}_L(-id/dt)S^{a}_K$ on
${\mathcal D}((-\infty, a))$ where $S^{a}_K:(-\infty,a)\rightarrow
L(E,[D(A)])$ is continuous, $S^{a}_K(t)=0,\; t\leq 0.$
\item[(b)'] $A$ generates a (EDSG)  $G$  so that
$G$ is of the form\\ $G=P_L(-id/dt)S_K$ on
${\mathcal{SE}}_{a}({\mathbb R})$, where $S_K: {\mathbb
R}\rightarrow L(E,[D(A)])$ is continuous, $S_K(t)=0,\; t\leq 0$ and
$e^{-at}||S_k(t)||\leq A|t|^k,$ for some $k>0$ and $A>0$.
\item[(c)] For every $a>0$, $A$ is the generator  of a local
non-degenerate $K_{a}$-convoluted semigroup
$(S^{a}_{K_{a}}(t))_{t\in [0,a)},$ where $K_{a}={\mathcal
L}^{-1}(\frac{1}{P^{a}_L(-i\lambda)})$ and $P^{a}_{L}$ is a
differential operator  such that for $0<a<b$ the
restriction of $P^{b}_{L}S^{b}_{K},$
on ${\mathcal D}((-\infty,a))$ is equal to $P^{a}_{L}S^{a}_{K}.$
\item[(c)']  $A$ is the generator of a global, exponentially bounded
non-degenerate $K$-convoluted semigroup $(S_{K}(t))_{t\geq 0},$
where $K={\mathcal L}^{-1}(\frac{1}{P_L(-i\lambda)})$.
\item[(d)] There exists an distribution fundamental solution  for
$A$, denoted by $G,$ with the property ${\mathcal N}(G)=\{0\}.$
\item[(d)'] There exists an exponential distribution fundamental solution
$G$ for $A$, with the property ${\mathcal N}(G)=\{0\}.$
\item[(e)] $\rho(A)\supset \Omega_{\alpha,\beta}=\{\xi+i\eta:\, \alpha\log(1+|\eta|)+\beta\}$ and
$$
||R(\lambda : A)||\leq C(1+|\lambda|)^k, \quad \lambda\in \Omega_{\alpha,\beta},
$$
for some $k>0$ and $ C>0$.
\item[(e)']   $\rho(A)\supset \{\lambda \in \mathbb{C} : Re \lambda>a\}$ and
$$
||R(\lambda : A)||\leq C(1+|\lambda|)^k ,\mbox{   } Re \lambda>a,
$$
for some $a, k>0$ and $ C>0$.
\end{itemize}
Then,
(a) $\Leftrightarrow$ (d); (a)' $\Leftrightarrow$ (d)'; (c)
$\Rightarrow$ (d); (c)' $\Rightarrow$ (d)'; (d) $\Leftrightarrow$ (e);
(d)' $\Leftrightarrow$ (e)'; $(a)'$ $\Rightarrow$ (c)'.\end{thm}
Note that the case of exponential distribution semigroup case is given parallel and it is denoted with $'$.
The proof of this theorem is similar (even simpler) with ~\cite[Theorem 2.7]{tica}, so we omit it.

\section{Stationary dense operators in locally convex spaces}
Following P. C. Kunstmann \cite{ku101}, we introduce the notion of a stationary dense operator in a sequentially complete locally convex space as follows.

\begin{defn}\label{pc-kunst}
A closed linear operator $A$ is said to be stationary dense iff
$$
n(A):=\inf\bigl\{k\in\mathbb{N}_0:D(A^m)\subseteq\overline{D(A^{m+1})}\text{ for all }m\geq k\bigr\}<\infty.
$$
\end{defn}

\begin{example}
Let $E=L^p({\mathbb R})$, $1\leq \leq\infty$. We consider the multiplication operator with maximal domain in $E$:
$$Af(x)=(x+ix^2)f(x),$$ $x\in{\mathbb R}$, $f\in E$.
It is clear that $A$ is dense and stationary dense if $1\leq p<\infty$. If $p=\infty$, then $A$ is not stationary dense. Indeed, the function $g(x)=\frac{1}{x^{2n}+1}$ belongs to $D(A^n)\backslash \overline{D(^{n+1})}$, for $n\in{\mathbb N}$.
\end{example}
The abstract Cauchy problem
\[(ACP_1):\left\{
\begin{array}{l}
u\in C([0,\tau):[D(A)])\cap C^1([0,\tau):E),\\
u'(t)=Au(t),\;t\in[0,\tau),\\
u(0)=x,
\end{array}
\right.
\]
where $A$ is a closed linear operator on $E$ and $0<\tau \leq \infty,$
has been analyzed in a great number of research papers and monographs (see e.g. \cite{d81},
\cite{komatsulocally}-\cite{knjigaho}, \cite{ko98},
\cite{ptica}-\cite{tica}, \cite{ku101}-\cite{li121},
\cite{me152} and \cite{ush}).
By a mild solution of problem $(ACP_1)$ we mean any continuous function $t\mapsto u(t;x),$ $t\in [0,\tau)$ such that $A\int^{t}_{0}u(s;x)\, ds=u(t;x)-x,$ $t\in [0,\tau).$

\begin{prop}\label{pc-kunst-isto}
Let $0<\tau \leq \infty $ and $n\in {\mathbb N}_{0}.$
\begin{itemize}
\item[(i)]
Suppose that the abstract Cauchy problem $(ACP_1)$ has a unique mild solution $u(t;x)$ for all $x\in D(A^{n}).$
Then $A$ is stationary dense and $n(A)\leq n.$
\item[(ii)] Suppose that $(S_{n}(t))_{t\in [0,\tau)}$ is a locally equicontinuous
$n$-times integrated semigroup generated by $A.$ Then $A$ is stationary dense and $n(A)\leq n.$
\end{itemize}
\end{prop}
\begin{proof}
One has to use the arguments given in that of \cite[Lemma 1.7]{ku101} (cf. also \cite[Remark 1.2(i)]{ku101}) and the fact that for any locally equicontinuous
$n$-times integrated semigroup $(S_{n}(t))_{t\in [0,\tau)}$ generated by $A$ the abstract Cauchy problem
$(ACP_1)$ has a unique mild solution for all $x\in D(A^{n}),$ given by
$
u(t;x)=S_{n}(t)A^{n}x+\sum^{n-1}_{j=0}\frac{t^{j}}{j!}A^{j}x,$ $t\in [0,\tau)$.
\end{proof}
\begin{lem}\label{lema0} Let $A$ be a closed operator in sequentially complete locally convex space and let $(\lambda_n)\in\rho(A)$ be a sequence such that $\lim_{n\rightarrow\infty}|\lambda_n|=\infty$ and there exist $C>0$ and $k\geq-1$ such that for every $p\in\circledast$, there exists $q\in\circledast$, such that $p(R(\lambda_n:A)x)\leq C|\lambda_n|^kq(x)$ for all $x\in E$ and $n\in\mathbb N$. Then $A$ is stationary dense with $n(A)\leq k+2$.\end{lem}
\begin{proof}
Let $x\in D(A^{k+1})$. Then for every $p\in\circledast$, there exists $q\in\circledast$, such that $p(\lambda_n R(\lambda_n:A)x)\leq \|R(\lambda_n:A)\| q(Ax)$. Now
for $x\in D(A^{k+2})$, it follows $\lambda_n R(\lambda_n:A)x\in D(A^{k+3})$ for all $\mathbb N$. Furthermore $$p(\lambda_n R(\lambda_n:A)x-x)=p(R(\lambda_n:A)Ax)\leq \frac{C'}{|\lambda_n|}q(Ax).$$ Hence, $$x=\lim\limits_{n\longrightarrow\infty}\lambda_n R(\lambda_n:A)x$$ and $x$ belongs to the closure of $D(A^{k+3})$, which means that $A$ is stationary dense and $n(A)\leq k+2$.
\end{proof}
\begin{example} The basic idea for this example is from \cite{ku101}. Let $E$ be a sequentially complete locally convex space and $A$ is a closed linear operator in $E$ which is not stationary dense  and there exist $C,k>0$ (for all $k>0$ there exist $C_k>0$) such that $$p(R(\lambda:A)x)\leq C e^{M(k|\lambda|)}p(x),\Re\lambda\geq0.$$
We construct inductively an increasing sequence $(M_n)_{n\in{\mathbb N_0}}$, $M_0=1$, and satisfying $(M.1)$, $(M.3)'$ and for all $\lambda\geq0$ $$\sum\limits_{p=0}^{\infty} \frac{\lambda^p}{M_p}\leq C_k e^{M(k|\lambda|)}.$$
The space $E_{M_n}$  is defined as $$E_{M_n}=\{f\in{\mathcal C}^{\infty}([0,1],E) : \|f\|_{M_n,h}=\sup\limits_{n}\frac{p(f^{(n)})}{M_n h^n}<\infty, \forall n\in\circledast\}.$$ Let $A=-\frac{d}{ds}$, with domain $D(A)=\{f\in E_{M_n} : f'\in E_{M_n},\, \, f(0)=0\}$. Let $g(t)=R(\lambda:A)f(t)=e^{-\lambda t}\int\limits_0^t e^{\lambda s}f(s)\, ds=\int\limits_0^t e^{-\lambda(t-s)}f(s)\, ds$. Then $g(0)=0$, $g$ is smooth function and $\lambda g+g'=f$. By $$g^{(n)}=f^{(n-1)}+\sum\limits_{k=1}^{n-1}(-\lambda)^k f^{(n-1-k)}+(-\lambda)^ng$$ and $p(R(\lambda:A)f)\leq p(f)$ we have
$$\frac{p(g^{(n)})}{M_n h^n}\leq\frac{p(f^{(n-1)})}{M_{n-1}h^{n-1}}+\sum\limits_{k=1}^{n-1}\frac{|\lambda|^k}{M_kM_{n-1-k}h^{n-1-k}} p(f^{(n-1-k)})+ \frac{|\lambda|^n}{M_nh^n}p(f).$$
Hence,
$$\|g\|_{M_n,h}\leq\sum\limits_{k=0}^{\infty}\frac{|\lambda|^k}{M_k}\|f\|_{M_{n},h}\leq Ce^{M(k|\lambda|)}\|f\|_{M_n,h},$$ which means that $g\in E_{M_n}$. Moreover, $g\in D(A)$ and $\|R(\lambda:A)\|\leq C e^{M(k|\lambda|)}$ and $A$ is a generator of an ultradistribution semigroup.\\
The domain of the operator $A^k$ is $D(A^k)=\{f\in E_{M_n}: \forall j\in\{0,1,2,...,k-1\}, f^{(j-1)}\in E_{M_n}, f^{(j)}(0)=0\}$ and $\overline{D(A^k)}\subset\{f\in E_{M_n}. \forall j\in\{0,1,2,...,k-1\}, f^{(j)}(0)=0\}$, $k\in{\mathbf N_0}$. We consider the function $f(t)=\frac{t^k}{k!}$. Then $f\in D(A^k)$, but not in $\overline{D(A^{k+1})}$, since $f^{(k)}(0)=1$. Hence $A$ is not stationary dense operator. \end{example}

We say that the operator $A$ satisfy the condition $(EQ)$ if:\\
$A_{\infty}$ is a generator of an equicontinuous semigroup $T_{\infty}(t)$ in $D_{\infty}(A)$, i.e. for every $p\in\circledast$ there exists $q\in\circledast$ and $C$ such that
$$p(T_{\infty}(t)x)\leq C q(x),$$ for every $x\in D_{\infty}(A)$.\\

Using the results given in \cite{ku101}, ~\cite[Theorem 4.1]{ush} we can state similar results in our setting (E is sequentially complete locally convex space). Assume that $A$ is stationary dense, satisfies $(EQ)$, $n=n(A)$ and $F$ is the closure of $D(A^n)$ in $E$.
\begin{lem}\label{lema1234}
\begin{itemize}
\item[a)]$A_F$ is densely defined in $F$, where $A_F$ means the restriction of the operator $A$ on $F$;
\item[b)] \label{lema2} $\rho(A;L(E))=\rho(A_F;L(F))$ for all $\lambda\in\rho(A)$. Additionally for all $x\in E$ and $p\in\circledast$, there exist $n\in{\mathbb N}$, $C>0$ such that
$$p(R(\lambda:A)x)\leq C(1+|\lambda|)^n(p(R(\lambda:A_F)x)+1);$$
\item[c)]\label{lema3} The Fr\'echet spaces $D_{\infty}(A)$ and $D_{\infty}(A_F)$ coincide topologically and $A_{\infty}=(A_F)_{\infty}$;
\item[d)]$n(A)=\inf\{k\in{\mathbf N_0}\, :\, \overline{D(A^k)}\subset D_{\infty}(A)\}.$\end{itemize}
\end{lem}
\begin{proof}\begin{itemize}
\item[a)] It is obvious since $D(A^n)$ is dense in $F$ and $D(A^n)\subset D(A_F)$.
\item[b)]
Since $D(A^n)\subset D(A)$, $F$ is invariant under resolvent of $A$, we obtain $\rho(A;L(E))\subset\rho(A_F;L(F))$. Now, let $\lambda\in\rho(A_F;L(F))$. Then $\lambda-A$ is injective, so $R(\lambda:A)$ is an extension of $R(\lambda:A_F)$. Let $\mu\in\rho(A)$. By $R(\lambda:A)=(\mu-A)^nR(\lambda:A)R(\mu:A)^n$, we obtain that $\lambda\in\rho(A:L(E))$.\\

It holds, for $x\in E$ and all $p\in\circledast$, $$p(R(\lambda:A)x)\leq p_n(R(\mu:A)^nx)p_n(R(\lambda:A_F)x)p(R(\mu:A)^nx)$$
where $p_n(x)=\sup\limits_{p\in\circledast}\sum\limits_{i=1}^{n}p(A^ix)$. Note that $(D(A^n),p_n)$ is a Banach space. For $x\in D(A^n)$ we have
$$p_n(R(\lambda:A_F)x)=\sup\limits_{p\in\circledast}\sum\limits_{i=0}^{n}p(A^iR(\lambda:A_F)x)=$$
$$=\sup\limits_{p\in\circledast}\sum\limits_{i=0}^{n}p(({\lambda}^k+A^k-{\lambda}^k)R(\lambda:A_F)x)\leq$$
$$\leq\sup\limits_{p\in\circledast}\sum\limits_{i=0}^{n}{\Biggl(}{|\lambda|}^k
p(R({\lambda}:A_F)x)+
\sum\limits_{j=0}^{i-1}{|\lambda|}^j p(A^{i-1-j}x){\Biggr)}.$$
The last one inequality, together with the previous one gives the statement of the lemma.

\item[c)]It holds $D(A^{n+k})\hookrightarrow D((A_F)^k)\hookrightarrow D(A^k)$, which one holds for all $k\in{\mathbb N}_0$. Then $$\bigcap\limits_{k=1}^{\infty}D(A^{n+k})\hookrightarrow\bigcap\limits_{k=1}^{\infty}D(A_F^k)
\hookrightarrow\bigcap\limits_{k=1}^{\infty}D(A^k),$$ which gives that $D_{\infty}(A)$ and $D_{\infty}(A_F)$ coincide topologically and $A_{\infty}=(A_F)_{\infty}$.
\item[d)] From the previous lemma, $D_{\infty}(A)=D_{\infty}(A_F)$ and $D(A^{n+1})\subset D(A_F)$, we obtain the conclusion of the lemma.\end{itemize}\end{proof}

\begin{thm}\label{eqeq} Let $A$ be a stationary dense operator in $E$ with non-empty resolvent. Then $\rho(A;L(E))=\rho(A_{\infty};L(D_{\infty}(A)))$.\end{thm}
\begin{proof}
By Lemma \ref{lema1234} a) we have that $A_F$ is dense in $F$ and by b) from the same lemma it has non-empty resolvent. Then using the proof of ~\cite[Theorem 2.3]{ku101} and Lemma \ref{lema1234} c), we obtain that $\rho(A;L(E))=\rho(A_{\infty};L(D_{\infty}(A)))$.
\end{proof}
Next example show us if $A$ is not stationary dense operator in $E$, then the conclusion of Theorem \ref{eqeq} does not hold. \begin{example}
Let we define the space $S_j$ as $$S_j=\{\varphi\in{\mathcal C}^{\infty}({\mathbb R})\, :\, p_j(x)=\sup\limits_{\alpha+\beta\leq j}\|x^{\beta}D^{\alpha}\varphi(x)\|_{L^2({\mathbb R})}<\infty\}.$$ Then the test space for tempered distributions $S({\mathbb R})$ can be defined as ${\mathcal S}({\mathbb R})=\lim\limits_{j\rightarrow\infty}\mbox{proj} S_j$. Let $E={\mathcal S}({\mathbb R})$, ($E$ is a Fr\'echet space as a projective limit of Banach spaces, so $E$ is a sequentially complete locally convex space). Define $A=-\frac{d}{dt}$ on $E$ with domain $D(A)=\{f\in E\, :\, f(0)=0\}$. The operator $A$ is not stationary dense on $E$. Note that $D_{\infty}(A)=\{0\}$ and $\rho(A_{\infty})=\{\lambda\in{\mathbb C}\, :\, \lambda\neq 0\}$. For $f\in E$, $\lambda\in{\mathbb C}$, and $\Re\lambda>s$, we have $$(\lambda-A)^{-1}f=\int\limits_0^{\infty}e^{-(\lambda-s)t}f(t)\, dt$$ belongs in $E$. Then $\rho(A)=\{\lambda\in{\mathbb C}\, :\, \Re\lambda>s\}$. Therefore, we obtain that the conclusion of Theorem \ref{eqeq} does not hold. The same conclusion can be made in the ultradistribution case.
Consider the spaces $$E_{(h)}=\{f\in C^{\infty}({\mathbb R})\, :\, f(0)=0,\, p_n(f)=\sup\limits_{k\in{\mathbb N}_0}\sup\limits_{t\geq k}\frac{h^n|{f}^{(n)}(t)|}{M_n}<+\infty\}$$
$$E_{\{h\}}=\{f\in C^{\infty}({\mathbb R})\, :\, f(0)=0,\, p_n(f)=\sup\limits_{k\in{\mathbb N}_0}\sup\limits_{t\geq k}\frac{h^n|{f}^{(n)}(t)|}{M_n\prod\limits_{i=1}^nr_i}<+\infty\}$$ for $(r_i)$ monotonically increasing positive sequence and $(M_n)$ satisfying $(M.1)$ and $(M.3)'$. The spaces $E_{(h)}$ and $E_{\{h\}}$ are Fr\'echet spaces. Let $A=-\frac{d}{ds}$ with domain $D(A)=\{f\in{E_h}\, : \, f(0)=0, Af\in E_h\}$, where $E_h$ stands for both spaces. Let $p_n\in\circledast$ be a seminorm in $E_{\{h\}}$.
The previous consideration in distribution case for $E={\mathcal S}({\mathbb R})$ is similar and more simple then the  case with spaces $E_{(h)}$ and $E_{\{h\}}$.
\end{example}
\begin{thm}\label{kor1} Let $A$ be a stationary dense operator in a sequentially complete locally convex space $E$, $n=n(A)$ and $F=\overline{D(A^n)}$. Then $A$ generates a distribution semigroup in $E$ if and only if $A_F$ generates a distribution semigroup in $F$.\end{thm}
\begin{proof}The proof of this theorem is direct consequence of Theorem \ref{strucdis}, Lemma \ref{lema0} and Lemma \ref{lema1234} b). \end{proof}
Following two results are due to T. Ushijima and we can restate it in locally convex case as following.
\begin{thm}\label{pom1} Let $E$ be a sequentially complete locally convex space and $A$ be a closed operator with dense $D(A^{\infty})$. Then $A$ has the property $(EQ)$ if and only if there exists logarithmic region $\Omega_{\alpha,\beta}$ and $k\in{\mathbb N}$ and $C>0$ such that $$ ||R(\lambda : A)||\leq C(1+|\lambda|)^k, \quad \lambda\in \Omega_{\alpha,\beta}.$$
\end{thm}
\begin{thm} Let $E$ be a sequentially complete locally convex space and $A$ linear operator on $E$. The following conditions are equivalent:
\begin{itemize}
\item[i)] $A$ is the generator of a distribution semigroup $G$;
\item[ii)] $A$ is well-posed and densely defined;
\item[iii)] $A$ satisfy the condition $(EQ)$;
\item[iv)] $A$ is densely defined, and there exist an adjoint logarithmic region $\Omega_{\alpha,\beta}$ and $k\in{\mathbb N}$ and $C>0$ such that
$$ ||R(\lambda : A)||\leq C(1+|\lambda|)^k, \quad \lambda\in \Omega_{\alpha,\beta}.$$
\end{itemize}
\end{thm}
\begin{proof}
We will prove $i)\Rightarrow ii)\Rightarrow iii)\Rightarrow iv)\Rightarrow i)$. The statements $i)\Leftrightarrow ii)$ and $iv)\Rightarrow i)$ follow from Theorem \ref{strucdis} and $iii)\Rightarrow iv)$ follows by Theorem \ref{pom1}  It remains to show that $i)\Rightarrow iii)$ and $iv)\Rightarrow ii)$.\\
$i)\Rightarrow iii)$: By the definition of the distribution semigroup, we can conclude that $D(A^{\infty})$ contains
${\mathcal R}(G)$. Since ${\mathcal R}(G)$ is dense in $E$, $D(A^{\infty})$ is dense in $E$. By the results in the third section from \cite{ush1}, we obtain that $$({\boldsymbol\lambda}-{\mathbf A}_{\infty})^{-1}(1\otimes x)(\hat{\varphi})=G(\varphi)x,$$ for any $\varphi\in\DD$ and $x\in D_{\infty}(A)$. Note that with $({\boldsymbol\lambda}-{\mathbf A}_{\infty})^{-1}$ is denoted the generalized resolvent. Let $f_{\lambda}(t)=\mu(t)e^{-\lambda t}$, where $\mu\in\DD$ and $\mu(t)=1$ for $|t|\leq1$. Define the operator $R_{\infty}(\lambda)=G(f_{\lambda})$. Since $G\in\DD'_0(L(E))$, for all $p\in\circledast$, there exists $q\in\circledast$, $k\in{\mathbb N}$, $C>0$ such that for all $x\in E$, $\lambda\in{\mathbb C}$, $\Re\lambda\geq0$ $$p(R_{\infty}(\lambda)x)\leq C (1+|\lambda|)^kq(x).$$
Now, for any $x\in D_{\infty}(A)$, $\varphi\in\DD$, $a>0$,

$$\frac{1}{i}\int\limits_{a-i\infty}^{a+i\infty}\hat{\varphi}(\lambda)R_{\infty}(\lambda)x\, d\lambda=\int\limits_{a-i\infty}^{a+i\infty}{\Bigl(}\frac{1}{2\pi i}\int\limits_{-\infty}^{\infty}\varphi(t) e^{t\lambda}\, dt{\Bigr)}G(f_{\lambda})\, d\lambda\, x=$$

$$=G{\Bigl(}\mu(s)\frac{1}{2\pi i}\int\limits_{a-i\infty}^{a+i\infty}\, d\lambda\int\limits_{-\infty}^{+\infty}\varphi(t)e^{(t-s)\lambda}\, dt{\Bigr)}x=G(\mu\cdot\varphi)x=$$
$$=G(\varphi)x=({\boldsymbol{\lambda}}-{\textbf A}_{\infty})^{-1}(1\otimes x)(\hat{\varphi}).$$

Again, using third section in \cite{ush1}, we obtain $$\int_0^{\infty}\varphi(t)T_{\infty}(t)x\, dt=G(\varphi)x,$$ for all $\varphi\DD$ and $x\in D_{\infty}(A)$, which gives $iii)$.
$v)\Rightarrow ii)$: It is same like in $(v)\Rightarrow (ii)$ of ~\cite[Theorem 4.1]{ush}.
\end{proof}

\begin{thm}\label{thmmain} Let $A$ be a closed operator in $E$. Then $A$ generates a distribution semigroup in $E$ if and only if $A$ is stationary dense and $A_{\infty}$ generates an equicontinuous semigroup in $D_{\infty}(A)$.\end{thm}
\begin{proof}
Let $A$ generates distribution semigroup in $E$. By simpler version of Theorem \ref{strucdis} and Lemma \ref{lema0}, $A$ is stationary dense.
Now, let $n=n(A)$ and $F=\overline{D(A^n)}$. Then by Lemma \ref{lema1234} a) and Theorem \ref{kor1} $A_F$ generates a dense distribution semigroup in $F$. Then $(A_F)_{\infty}$ generates an equicontinuous semigroup in $D_{\infty}(A_F)$. By Lemma \ref{lema1234} c) follows the conclusion.\\
Opposite direction. Let $A$ is stationary dense and $A_{\infty}$ generates an equicontinous semigroup and $n=n(A)$ and $F=\overline{D(A^n)}$. Then $A_F$ generates a distribution semigroup in $F$. By Theorem \ref{kor1} we obtain that $A$ generates a distribution semigroup in $E$.
\end{proof}
\begin{thm} Let $A$ be a closed operator in $E$. Then $A$ generates an exponential distribution semigroup in $E$ if and only if $A$ is stationary dense and $A_{\infty}$ generates a quasi-equicontinuous semigroup in $D_{\infty}(A)$.\end{thm}
The proof is direct consequence by the results of exponential distribution semigroups listed before.

\end{document}